\documentclass[12pt]{amsart}
\usepackage{amssymb}
\usepackage[all]{xy}
\usepackage[raiselinks=false,colorlinks=true,citecolor=blue,urlcolor=blue,linkcolor=blue,bookmarksopen=true,pdftex]{hyperref}
\newtheorem{theorem}{Theorem}[section]
\newtheorem{lemma}[theorem]{Lemma}
\theoremstyle{definition}
\newtheorem{definition}[theorem]{Definition}

\newtheorem{examples}[theorem]{Examples}
\newtheorem{proposition}[theorem]{Proposition}
\newtheorem{corollary}[theorem]{Corollary}

\newtheorem{remark}[theorem]{Remark}

\def\aut{\operatorname{Aut}}
\def\out{\operatorname{Out}}
\def\homeo{\operatorname{Homeo}}
\def\supp{\textrm{supp}}
\def\hom{\operatorname{Hom}}
\def\ab{\textrm{ab}}
\def\fix{\textrm{Fix}}
\begin{document}

\title[Sigma theory and twisted conjugacy-II] {Sigma theory and twisted conjugacy-II: Houghton groups and pure 
symmetric automorphism groups}
\author[ D. L. Gon\c{c}alves] {Daciberg L. Gon\c{c}alves}
\address{Departametno de Matem\'atica - IME, Universidade de S\~ao Paulo\\ Caixa Postal 66.281 - CEP 05314-970, S\~ao Paulo - SP, Brasil}
\email{dlgoncal@ime.usp.br}
\author[P. Sankaran]{Parameswaran Sankaran}
\address{The Institute of Mathematical Sciences\\
CIT Campus, Taramani,\\ 
Chennai 600113, India}
\email{sankaran@imsc.res.in}

\subjclass[2010]{20E45, 20E36.\\
Keywords and phrases: Twisted conjugacy, Reidemeister number, sigma theory, Houghton groups, infinite symmetric group, pure symmetric automorphism groups.}

\maketitle
\noindent
{\bf Abstract} {\it Let $\phi:\Gamma\to \Gamma$ be an automorphism of a group $\Gamma$.    
We say that 
$x,y\in \Gamma$ are in the same $\phi$-twisted conjugacy class and write $x\sim_\phi y$ if there exists an element $\gamma\in \Gamma$ such that $y=\gamma x\phi(\gamma^{-1})$.  This is an equivalence relation on $\Gamma$ called the $\phi$-twisted conjugacy.  Let $R(\phi)$ denote the number of  $\phi$-twisted conjugacy classes in $\Gamma$.  If $R(\phi)$ is infinite for all $\phi\in \aut(\Gamma)$, we say that $\Gamma$ has the $R_\infty$-property.    

The  purpose of this note is to show that the symmetric group $S_\infty$, the Houghton groups and 
the pure symmetric automorphism groups have the $R_\infty$-property.  We show, also, that the Richard Thompson group $T$ has the $R_\infty$-property.  We obtain a general result establishing the $R_\infty$-property of finite direct product of 
finitely generated groups.  

This is a sequel to an earlier work by 
Gon\c{c}alves and Kochloukova, in which it was shown, using the sigma-theory 
due to Bieri-Neumann-Strebel, that for most of the groups $\Gamma$ considered here, $R(\phi)=\infty$ where $\phi$ varies in a finite index subgroup of the automorphisms of $\Gamma$.  }

\section{Introduction} 

Let $\Gamma$ be a group and let $\phi:\Gamma\to \Gamma$ be an endomorphism.  Then $\phi$ determines an action $\Phi$ of 
$\Gamma$ on itself where, for  $\gamma\in \Gamma$ and $x\in \Gamma$, we have $\Phi_\gamma(x)=\gamma x \phi(\gamma^{-1})$.  The orbits of this action are called the $\phi$-twisted conjugacy classes. 
We write $x\sim_\phi y$ if $x$ and $y$ are in the same $\phi$-twisted conjugacy class.
Note that when $\phi$ is the identity automorphism, the orbits are the usual conjugacy classes of $\Gamma$.   We denote by $\mathcal{R}(\phi)$ the set of all $\phi$-twisted conjugacy classes and by $R(\phi)$ the cardinality $\#\mathcal{R}(\phi)$ of $\mathcal{R}(\phi)$.   
We say that $\Gamma$ has the $R_\infty$-property if $R(\phi)=\infty,$ that is if $\mathcal{R}(\phi)$ is infinite, for every automorphism $\phi$ of $\Gamma$.   

The problem of determining which groups have 
the $R_\infty$-property---more briefly the $R_\infty$-problem---has attracted the attention of many researchers 
after it was discovered that all non-elementary Gromov-hyperbolic groups have the $R_\infty$-property.  See \cite{ll} and \cite{felshtyn}.  It is particularly interesting when the 
group in question is finitely generated or countable. 
The notion of twisted conjugacy arises naturally in fixed point theory, representation theory, algebraic geometry and number theory.    In recent years the $R_\infty$-problem has emerged as an active research area. 
 
Recall that Houghton introduced a family of groups $H_n, n\ge 2,$ defined as follows:  Let $M_n:=\{1,2,\ldots, n\}\times \mathbb{N}$.    The group $H_n$ consists of all bijections $f:M_n\to M_n$ such that there exists integers $t_1,\ldots, t_n$ 
such that $f(j,s)=(j,s+t_j)$  for all $s\in \mathbb{N}$ sufficiently large and all $j\le n$.  Note that necessarily $\sum_{1\le j\le n}t_j=0$. 
Let $Z=\{(t_1,\ldots,t_n)\mid\sum_{1\le j\le n}t_j=0\}\subset \mathbb{Z}^n \}\cong \mathbb{Z}^{n-1}.$  
One has a surjective homomorphism 
$\tau:H_n\to Z\cong\mathbb{Z}^{n-1}$ sending $f$ to its {\it translation part} $(t_1,\ldots, t_n)$ (with notation as above).  
It is easily verified 
that $\tau$ is surjective with kernel the group of all {\it finitary} permutations of $M_n.$  
K. S. Brown \cite{brown} showed that $H_n$ is finitely presented for $n\ge 3$ and that 
it is $\textrm{FP}_{n-1}$ but not $\textrm{FP}_n$.   Note that the above definition of $H_n$ makes sense 
even for $n=1$ and we have $H_1\cong S_\infty$.   However, we treat the group $S_\infty$ separately 
and we shall always assume that $n\ge 2$ while considering the family $H_n$.  

Next we recall the group $G_n$, the group of {\it pure symmetric automorphisms} of the free group $F_n$ of rank $n\ge 2$.   Fix a basis $x_k, 1\le k\le n$, of $F_n$.  Denote by  
$\alpha_{ij}\in \aut(F_n), 1\le i\ne j\le n,$ the automorphism defined as 
$x_i\mapsto x_jx_ix_j^{-1}, x_k\to x_k, 1\le k\le n, k\ne i$.   The group $G_n$ is the subgroup of $\aut(F_n)$ 
generated by $\alpha_{ij}, 1\le i\ne j\le n$.  McCool showed that $G_n$ is finitely presented where 
the generating relations are: \\
(i) $[\alpha_{ij},\alpha_{kl}]=1,$ whenever $i,j,k,l$ are all different; \\(ii)  $[\alpha_{ik}, \alpha_{jk}],=1$ and $[\alpha_{ij}\alpha_{kj},\alpha_{ik}]=1$ whenever $i,j,k$ are all different.   

It was shown by Gon\c{c}alves and Kochloukova \cite{gk} that $R(\phi)=\infty$ for all $\phi$ in a finite index 
subgroup of the group of all automorphisms of $\Gamma$ where $\Gamma=H_n,G_n.$  
Our main result is the following theorem.  We give two proofs for the case of Houghton groups, neither 
of which use $\Sigma$-theory.  However we still need to use the results of \cite{gk} in the case of $G_n$.   

\begin{theorem}  \label{main} The following groups have the $R_\infty$-property:\\
(i) The group $S_\infty$ of finitary permutations of $\mathbb{N}$, \\ (ii) the Houghton groups $H_n, n\ge 2,$ and,\\  
(iii) the group $G_n, n\ge 2,$ of pure symmetric automorphisms of a free group of rank $n$.
\end{theorem}

Recall that Richard Thompson constructed three groups finitely presented infinite groups $F\subset T\subset V$ around 1965 
and showed that $T$ and $V$ are simple.  The groups $F, T,$ and $V$ arise as certain homeomorphism groups 
of the reals, the circle, and the Cantor set respectively.   Since then these constructions have been 
generalized by G. Higman \cite{higman}.  See also K. S. Brown \cite{brown}, R. Bieri and R. Strebel \cite{bs}, and M. Stein \cite{stein}.   For an introduction to the Thompson groups $F, T, V$ see \cite{cfp}. 

\begin{theorem}
The Richard Thompson group $T$ has the $R_\infty$-property.
\end{theorem}

As the group $T$ is simple, $\Sigma$-theory yields no information about the $R_\infty$-property.
The above theorem was first proved by Burillo, Matucci, and Ventura \cite{bmv}. Shortly 
thereafter, Gon\c{c}alves and Sankaran \cite{gs} also independently obtained the same result.

In section \S2 we make some preliminary observations concerning  the $R_\infty$-property which will be 
needed for our purposes.   
Theorem \ref{main} will be established in \S3.   The $R_\infty$-property of the 
group $T$ will be proved in \S4.  In \S5 we consider the $R_\infty$-property finite direct product of groups 
and obtain a strengthening of a result of Gon\c{c}alves and Kochloukova \cite{gk}.

This is a sequel to the paper \cite{gk} by Gon\c{c}alves and Kochloukova.
We  reassure the reader that this paper can be read independently of it.  Although results from \cite{gk} are used,  we develop our own proof techniques to go forward. 

{\it If $f:X\to Y$ is a map of sets, we shall always write the argument to the right of $f$; thus $f(x)$ denotes 
the image of $x\in X$ under $f$. }

\section{Preliminaries}
We begin by recalling some general result concerning twisted conjugacy classes of an automorphism 
of a group and that of its restriction to a normal subgroup.   We obtain a criterion for a periodic automorphism 
to have infinitely many twisted conjugacy classes. We shall also briefly recall the notion of the Bieri-Neumann-Strebel 
invariant and give its known description in the case of Houghton groups and the pure symmetric automorphism groups.

\subsection{Addition formula} 
The following lemma may be found in \cite[\S2]{gw}.  For any element $g\in G$, we shall denote by $\iota_g$ the inner automorphism $x\mapsto 
gxg^{-1}$ of $G$.   When $N$ is a normal subgroup of $G$,  we shall abuse notation and denote by the same symbol 
$\iota_g$ the automorphism of $N$ got by restriction of $\iota_g$ to $N.$

\begin{lemma}\label{additionformula} 
Suppose that we have a commutative diagram of homomorphisms of groups 
where the vertical arrows are isomorphisms and horizontal rows are short exact sequence: 
\[ \begin{array}{ccccccccc}1 &\to& N&\stackrel{i}{\to}& G&\stackrel{p}{\to}& G/N&\to &1\\
 &&\downarrow \theta' &&\downarrow\theta&&\downarrow\bar{\theta} &&\\
 1 &\to& N&\stackrel{i}{\to}& G&\stackrel{p}{\to}& G/N&\to &1.\\
\end{array}
\]
Then: \\
(i) One has an exact sequence of (pointed) sets  $\mathcal{R}(\theta')\stackrel{i_*}{\to }\mathcal{R}(\theta)\stackrel{p_*}{\to} \mathcal{R}(\bar{\theta}) \to \{0\}$.   That is, $p_*$ is surjective and 
$\operatorname{Im}(i_*)=p_*^{-1}(\{N\})$.  \\
(ii) {\em (Addition Formula):}  Suppose that $R(\bar{\theta})<\infty$ and that $\fix(\iota_{\alpha N}\circ\bar{\theta})=\{N\}$ for all $\alpha\in G$.  Then $R(\theta)<\infty$ 
if and only if $R(\iota_\alpha\theta')<\infty$ for all $\alpha\in G$.  Moreover, the following {\em addition formula}  
holds if $R(\theta)<\infty$:
\[ R(\theta)=\sum_{[\alpha N]\in \mathcal{R}(\bar{\theta})} R(\iota_\alpha \theta').\]
\end{lemma}

\begin{proof} We prove only (ii), assertion (i) being well-known and easy.

Let $\alpha\in G$, and $x,y\in N$. 
Suppose that $x,y$ are in the same $\iota_\alpha\theta$-twisted conjugacy class when regarded as elements of $G$.  
That is $y=zx\iota_\alpha \theta(z^{-1})$ for some $z\in G.$    
Applying the projection $p:G\to G/N$ and denoting $xN$ by $\bar{x}$ we obtain $\bar{1}=\bar{y}=
\bar{z}\iota_{\bar{\alpha}}\bar{\theta}(z^{-1})$ and so $\bar{z}\in \fix(\iota_{\bar{\alpha}}\bar{\theta})=\{\bar{1}\}$.  
Hence $z\in N$ and so $y$ and $x$ are $\iota_\alpha\theta'$-twisted conjugates.  Thus  
we see that  
distinct $\iota_\alpha\theta'$-conjugacy classes map to {\it distinct} $\iota_\alpha\theta$- conjugacy classes.  This implies that $R(\theta')<\infty$ if $R(\theta)<\infty$.  

Also, $y=zx\iota_\alpha \theta(z^{-1})$ 
implies that $y\alpha=z(x\alpha)\theta(z^{-1})$ and so $y\alpha, x\alpha\in G$ are $\theta$-twisted conjugates.  Thus translation on the right by $\alpha$ of $\theta'$-twisted conjugacy classes of $N$ are contained $\theta$-twisted 
conjugacy classes.  

Suppose that $x\alpha, x\beta$ are in the same $\theta$-twisted conjugacy class where $x\in N, \alpha, \beta\in G$. Then 
there exists a $u\in G$ such that $x\beta=ux\alpha \theta(u^{-1})$.  Applying the projection we obtain that 
$\bar{\beta} =\bar{u} \bar{\alpha} \bar{\theta}(\bar{u}^{-1})$.  Thus $\bar{\alpha}$ and $\bar{\beta}$ are 
$\bar{\theta}$-twisted conjugates.    This implies that the  
right translates of $\theta'$-twisted conjugacy classes by elements 
$\alpha,\beta$ belong to the same $\theta$-twisted conjugacy class only if that $\bar{\alpha}$ and $\bar{\beta}$ 
are $\bar{\theta}$-twisted conjugates. 
Conversely, if $\bar{\alpha}$ and $\bar{\beta}$ are $\bar{\theta}$ twisted 
conjugates, then by reversing the arguments, we see that the translates by $\alpha$ and $\beta$ of $\theta'$-twisted conjugacy classes 
are contained in the same $\theta$-twisted conjugacy class. This establishes the addition formula.
\end{proof} 

\begin{remark}\label{minusid}
 Note that if $G/N\cong \mathbb{Z}^n, n<\infty,$ and if $1$ is not an eigenvalue of the matrix of $\bar{\theta}$  
 with respect to a basis of $G/N$, then 
for any $\alpha\in G, \fix(\iota_{\alpha N}\circ \bar{\theta})=\fix(\bar{\theta})$ consists only of the trivial element. So the lemma implies that, if $R(\theta')=\infty$, then $R(\theta)=\infty$.
\end{remark}

\subsection{Periodic outer automorphisms} \label{periodic}

Let $\Gamma$ be a group with infinitely many conjugacy classes.  Then, for any automorphism $\phi:\Gamma\to \Gamma$, and any $g\in G$, $R(\phi)=R(\iota_g\circ \phi)$ where $\iota_g$ denotes the inner automorphism $x\mapsto gxg^{-1}$.  
Indeed it is readily seen that the $\phi$-twisted conjugacy classes are the same as the left translation by $g$ of the 
$\iota_g\circ\phi$-twisted conjugacy classes.  Thus $\Gamma$ has the $R_\infty$-property if and only if $R(\phi)=\infty$ for a set of coset representatives of $\out(\Gamma)=\aut(\Gamma)/Inn(\Gamma).$  We have the following lemma.  
Compare \cite{gs}. 

\begin{lemma} \label{finiteorder}
Let $\theta\in \aut(\Gamma)$ and let $n\ge 1$.   
Suppose that $\{x^n\mid x\in Fix(\theta)\}$ is not contained 
in the union of finitely many $\theta^n$-twisted conjugacy classes of $\Gamma$.  Then $R(\theta)=\infty$. 
\end{lemma}
\begin{proof}
Let $x\sim_\theta y$ in $\Gamma$ where $x,y\in \fix(\theta)$. Thus there exists an $z\in \Gamma$ such that 
$y=z^{-1}x\theta(z)$.  Applying $\theta^i$ both sides,  
we obtain $y=\theta^i(z^{-1})x\theta^{i+1}(z)$ as $x,y\in \fix(\theta)$.  
Write $\phi:=\theta^n$.
Multiplying these equations successively for $0\le i<n$,  we obtain  
\[ y^n=\prod_{0\le i<n} \theta^i(z^{-1})x\theta^{i+1}(z)
=z^{-1}x^n\theta^n(z)=z^{-1}x^n\phi(z).\]
That is, $y^n\sim_{\phi} x^n$.  
Our hypothesis says that there are infinitely many elements 
$x_k\in \fix(\theta), k\ge 1$, such that the $x_k^n$  are in 
pairwise distinct $\phi$-twisted conjugacy classes of $\Gamma$.   Hence 
we conclude that $R(\theta)=\infty$.  
\end{proof}

\begin{remark} \label{torsionfix}
When $\theta^n=\iota_\gamma$ is an inner automorphism, we see from the above lemma 
that $R(\theta)=\infty$ 
if $\{x^n\gamma\mid x\in Fix(\theta)\}$ is not contained in a finite union of conjugacy classes of $\Gamma$. 
If $\theta^n=id$, then $R(\theta)=\infty$ if $Fix(\theta)$ contains elements of order $k$ for arbitrarily large values of $k\in \mathbb{N}$. 
 \end{remark}

\subsection{$\Sigma$-theory of $H_n$ and $G_n$}\label{sigmath}

Bieri, Neumann, and Strebel \cite{bns} introduced, for any finitely generated group $\Gamma$, an invariant $\Sigma(\Gamma)$
which is a certain open subset---possibly empty---of the character sphere $S(\Gamma):=\hom(\Gamma,\mathbb{R})\setminus\{0\}/\mathbb{R}_{>0}$ where the action of the multiplicative group of positive reals is via scalar multiplication.   The automorphism group $\aut(\Gamma)$ acts on $S(\Gamma)$ where $\phi^*:S(\Gamma)\to S(\Gamma)$ 
is defined as $[\chi]\mapsto [\chi\circ \phi], [\chi]\in S(\Gamma)$, for $\phi\in\aut(\Gamma)$. 
This action preserves the subspace  
$\Sigma(\Gamma)$ and hence also its complement $\Sigma^c(\Gamma)$.   
If the image of the homomorphism $\eta: \aut(\Gamma)\to \homeo(\Sigma^c(\Gamma))$ is a finite group, then $K=\ker(\eta)$ is a finite index subgroup of $\aut(\Gamma)$ which fixes every character class in $\Sigma^c(\Gamma)$.  
 This happens, for example, if $\Sigma^c(\Gamma)$ is a non-empty  
finite set.
If $\Sigma^c(\Gamma)$ contains a {\it discrete} character class $[\chi]$, that is, a class 
represented by an character $\chi$ whose image $\chi(\Gamma)\subset \mathbb{R}$ is infinite cyclic, then it was observed by  Gon\c calves and Kochloukova \cite{gk} that the character $\chi$ itself is fixed by the action of $K$ on 
$\hom(\Gamma,\mathbb{R})$.     That is, $\chi\circ \phi=\chi$ for all $\phi\in K\subset \aut(\Gamma)$.  This easily 
implies that $R(\phi)=\infty$ by Lemma \ref{additionformula}(i), taking $G=\Gamma, N=\ker{\chi}, \theta=\phi$ in the 
notation of that lemma, so that $\bar{\theta}=id$.  

When $\Gamma=G_n, n\ge 3,$ the 
group of pure symmetric automorphisms of $F_n$, L. Orlandi-Korner \cite{o} has determined $\Sigma^c(\Gamma)$. 
When $\Gamma=H_n$, the Houghton group, Brown \cite{brown1} computed the set $\Sigma^c(\Gamma)$. 
Using these results, Gon\c calves and Kochloukova, showed that if $\Gamma$ 
is any one of the groups $H_n, n\ge 2, G_m, m\ge 3$, the image of $\eta:\aut(\Gamma)\to \homeo(\Sigma^c(\Gamma))$ is finite.

In the case of the Houghton group $H_n, n\ge 2$, it turns out that $\Sigma^c(H_n)$ is a finite set of discrete 
character classes $[\chi_j], 1\le j\le n$.  Explicitly, $\chi_j:H_n\to \mathbb{Z}$ may be taken to be 
$-\pi_i\circ \tau$ where $\tau:H_n\to Z$ is the translation part (see \S1) and $\pi_i:Z\to \mathbb{Z}$ is the restriction of to $Z\subset \mathbb{Z}^n$ 
of the $i$-th projection (see \cite{brown1}).  (Recall from \S1 that $Z=\{(t_1,\ldots, t_n)\in \mathbb{Z}^n\mid \sum_{1\le j\le n}t_j=0\}$.) 
Thus $\homeo(\Sigma^c(H_n))\cong S_n$ is finite and so is the image of $\eta:\aut(H_n)\to \homeo(\Sigma^c(H_n))$. 
As already remarked $R(\phi)=\infty$ for all $\phi\in \ker (\eta)$.     The lemma 
below will not be used in the sequel but included here for illustrative purposes.

\begin{lemma}
Suppose that $\eta(\phi):\Sigma^c(H_n)\to \Sigma^c(H_n)$ is not an $n$-cycle.   Then $R(\phi)=\infty$.
\end{lemma}
\begin{proof}  Since $\eta(\phi)$ is not an $n$-cycle, the orbit of $[\chi_1]$ under $\eta(\phi)$ consists of at most 
$n-1$ elements.  Since $\chi_1$ are all discrete,  
the orbit of $\chi_1\in \hom(H_n,\mathbb{R})$ consists of at most $n-1$ elements.  (In fact the orbit of $\chi_1$ 
is a subset of $\{\chi_j\mid 1\le j\le n\}$.)  Now the orbit sum $\lambda:=\sum_{1\le j\le k} \chi_1\phi^j$ is a {\it non-zero} character 
since any $n-1$ elements of $\chi_j, 1\le j\le n$ form a {\it basis} of $\hom(H_n,\mathbb{R})$.  It follows that, since 
$\phi^*(\lambda)=\lambda$, $R(\phi)=\infty$. 
\end{proof}

If $\phi^*:\Sigma^c(H_n)\to \Sigma^c(H_n)$ is an $n$-cycle, then the orbit sum is zero and the above 
argument fails.  In fact, it is easily seen that every possible permutation of $\Sigma^c(H_n)$ 
may be realized as $\eta(\phi)$ for some $\phi\in \aut(H_n)$, that is, $\eta:\aut(H_n)\to \homeo(\Sigma^c(H_n))\cong S_n$ 
is surjective.


\section{Proof of Theorem \ref{main}}

Let $X$ be an infinite set.   We will only be concerned with the case when $X$ is countably infinite.   
We shall denote by $S_\infty(X)$ the group of all finitary permutations of $X$, that 
is those permutations which fixes all but finitely many elements of $X$.  The group of {\it all} permutations 
of $X$ will be denoted by $S(X)$.  We shall denote $S(X)$ (resp. $S_\infty(X)$) 
simply by $S_\omega$ (resp. $S_\infty$) when $X$ is clear from the context. If $x=(x_k)_{k\in \mathbb{Z}}$ is a doubly infinite sequence in $X$ of pairwise distinct 
elements, we regard 
it as an element of $S(X)$ where $x(x_k)=x_{k+1}$ and $x(a)=a$ if $a\ne x_k~\forall k\in \mathbb{Z}$.  
Two such sequences $x=(x_k)$ and $y=(y_k)$ define the same permutation if and only if $y$ is a shift of $x$, that is,  
there exists an $n$ such that 
$x_k=y_{k+n}$ for all $k\in \mathbb{Z}$.   Thus, the sequence $x=(x_k)_{k\in \mathbb{Z}}$ is just the 
infinite cycle $x\in S(X)$.    
Any $f\in S(X)$ is uniquely expressible as a product of disjoint cycles.  Such an expression of $f$ is its {\it cycle decomposition}.   
The {\it cycle type} 
of an $f\in S(X)$ is the function $c(f):\mathbb{N}\cup\{\infty\}\to \mathbb{Z}_{\ge 0}\cup \{\infty\}$  where $c(f)(\alpha)$ is the number of 
$\alpha$-cycles in the cycle decomposition of $f$ if that number is finite, otherwise it is $\infty$ for $\alpha\in \mathbb{N}\cup\{\infty\}$.    As in the case 
$S_\infty(X)$,  if $f$ and $g$ have the same cycle type, then they are conjugate in $S(X)$.  
We need a criterion for $f$ and $g$ to be conjugate by an element of $S_\infty(X)$.

\begin{lemma} \label{transcycle}   
Let $x=(x_k)_{k\in\mathbb{Z}}, y=(y_k)_{k\in \mathbb{Z}}\in S_\omega(X)$ be two disjoint infinite cycles and let $(a,b)\in S_\infty$. \\
(i)  If $a=x_0, b=x_k, k>0,$
then $(a, b)x=uv$, a product of disjoint cycles $u=(u_j)_{j\in \mathbb{Z}}\in S_\omega, v\in S_\infty$, defined as 
\[u_j=\left\{\begin{array}{cc} x_j, & j<0,\\
x_{j+k}, &j\ge 0,\\ 
\end{array}
\right .
\]
and $v=(x_0,\ldots,x_{k-1})\in S_\infty$. 

(ii)  If $a=x_0, b=y_0$, then $(a,b)xy=uv$, where $u=(u_j)_{j\in \mathbb{Z}}, v=(v_j)_{j\in \mathbb{Z}}$ are disjoint infinite 
cycles defined as 
\[u_j=\left\{\begin{array}{cc} x_j, &j<0,\\
y_j, &j\ge 0,\\
\end{array}
\right .
\] and 
\[v_j=\left\{\begin{array}{cc} y_j, &j<0,\\
x_j, &j\ge 0. \\
\end{array} 
\right . 
\]~ \hfill $\Box$
\end{lemma}

If $k\in \mathbb{N}$, we denote by $\mathbb{N}_{>k}$ the set of all integers greater than $k$.      
Note that $S_\infty=\cup_{k\ge 2}S_k$ where $S_k$ is the subgroup 
consisting of permutations of $\mathbb{N}$ which fixes all $n>k$.   In particular, the group $S_\infty$ is generated  by 
transpositions $(i, i+1), i\ge 1$.    The alternating group $A_\infty$ equals the commutator subgroup $[S_\infty, S_\infty]$,  has index $2$ in $S_\infty$ and is simple.   
The conjugacy class of any element of $S_\infty$ is determined by its cycle type, as in the case 
of finite symmetric groups.  The group $S_\infty$ is a normal subgroup of $S_\omega=S(\mathbb{N})$.   
In particular, 
any bijection $f:\mathbb{N}\to \mathbb{N}$ defines an automorphism $\iota_f\in \aut(S_\infty)$ by restricting the inner 
automorphism determined by $f\in S_\omega$.  Moreover $\iota_f$ is the identity automorphism only if $f$ is equals the  identity map.  
The following lemma is perhaps well-known, although we could not find a reference for it.

\begin{lemma} \label{autsym} 
The homomorphism $\iota: S_\omega\to \aut(S_\infty)$ is an isomorphism of groups. 
\end{lemma}
\begin{proof}
By the discussion above, it only remains to show that 
$\iota$ is surjective.  Let $\theta:S_\infty\to S_\infty$ be an isomorphism.  
We claim that $\theta$ maps transpositions to transpositions.  Indeed, let $\tau=(1,2)$ so that $\theta(\tau)$ 
is an element of order $2$. Hence $\theta(\tau):=\xi$ is a product of disjoint transpositions $\xi=\tau_1.\ldots.\tau_k$. 
To establish our claim, we must show that $k=1$. 
(Since $\tau$ is an odd permutation, $k$ must be odd. This fact is not required here.)   Without loss of generality, we assume that $\tau_i=(i ,k+i), 1\le i\le k$, so that $\xi=(1,k+1)(2,k+2)\ldots (k,2k)$.

The 
centralizer of $\tau$ is the subgroup $Z(\tau)=\{\sigma\in S_\infty\mid \sigma\tau=\tau\sigma\}=
S_2\times S_\infty(\mathbb{N}_{\ge 3})\subset S_\infty.$    

The automorphism $\theta$ maps the $Z(\tau)$ isomorphically onto $Z(\theta(\tau))=Z(\xi)$.
On the other hand, if $\sigma\in S_\infty$, then $\sigma\xi\sigma^{-1}=\sigma\tau_1\ldots \tau_k \sigma^{-1}=(\sigma(1),\sigma(k+1))\ldots (\sigma(k),\sigma(2k))$.   Thus $\xi=\sigma \xi\sigma^{-1}$ holds if and only if $\sigma\in S_{2k}\times S(\mathbb{N}_{>2k})$ and 
$|\sigma(j+k)-\sigma(j)|=k$ for $1\le j\le k$.  That is, $Z(\xi)=Z_0\times S(\mathbb{N}_{>2k})$ where $Z_0\subset S_{2k}$ 
is the centralizer of $\xi\in S_{2k}.$  
Denote by $\beta:S_k\to S_{2k}$ the monomorphism $(\beta(h))(j)=h(j), \beta(h)(j+k) =h(j)+k, 1\le j\le k$.  Then we see that $\sigma\in Z_0$ if and only if $\sigma=\tau_{i_1}\ldots \tau_{i_r}\beta(h)$ for some $h\in S_k$ and $1\le i_1<\cdots<i_r\le k$; moreover the element $h$ and the (possibly empty) sequence $i_1,\ldots, i_r$ are determined uniquely by $\sigma$.   Consider the homomorphism $Z_0\to (\mathbb{Z}/2\mathbb{Z})^k\rtimes S_k$ defined as $\sigma\mapsto (t, h)$ where $t_j=1$ or $0$ according as $j$ belongs to $\{i_1,\ldots, i_r\}$ or not.  (Here the group $S_k$ operates on $(\mathbb{Z}/2\mathbb{Z})^k$ by permuting the coordinates.) It is easily seen that this is an isomorphism 
of groups.
In particular $Z_0$ has order $2^k.k!$.

Since $\theta$ is an automorphism, we must have $S_2\times S(\mathbb{N}_{>2})\cong Z(\tau)\cong Z(\theta(\tau))=Z(\xi)
\cong  Z_0\times S(\mathbb{N}_{>2k})$.  
Now $Z(\xi)$ has a normal subgroup of index $2^{k+1}.k!$, namely the alternating group on $\mathbb{N}_{>2k}$.  On the other hand, 
the only finite index proper normal subgroups of $Z(\tau)\cong S_2\times S_\infty$ are $1\times A_\infty$, which 
has index $4$, and  $S_2\times A_\infty$ 
and $1\times S_\infty$, each of which has index $2$.  It follows that $k=1$.
Thus $\theta((1, 2))=(1, 2)$.  
By the same argument $\theta(1, 3)=(i , j)$.  We assert that either $i\in \{1, 2\}$ or $j\in \{1,2\}$.  In fact if $\{i, j\}\cap \{1,2\}=\emptyset$ or if $\{i, j\}=\{1,2\}$, then $\theta((1, 2))$ and $\theta((1, 3))$ commute, a contradiction.  
Suppose that $i\in \{1,2\}$ and $j\notin \{1,2\}$.  We set $f(1)=i$.  In general, $f(k)$ is defined to be the unique number $l$ where $\theta((k,  k+1))=(l, a), \theta
((k,k+2))=(l, b)$.    Then $f:\mathbb{N}\to \mathbb{N}$ is a monomorphism.  It is a surjection since $\theta$ is. 
It is readily seen that $\theta=\iota_f$.  
\end{proof}

\begin{corollary} \label{autgen} 
{\em 
Suppose that $S_\infty$ is a characteristic subgroup of a group $H$ contained in $S_\omega$.  
Then the automorphism group of $H$ is isomorphic to the normalizer of $N(H)$ of $H$ in $S_\omega$. 
In particular, every automorphism of $H$ is the restriction to $H$ of a unique inner automorphism of 
$S_\omega$.}
\end{corollary}
\begin{proof}
We shall denote by the same symbol $\iota_f$ to denote the conjugation by $f\in S_\omega$ or its restriction 
to any subgroup normalized by $f$.

It is evident that $\iota:N(H)\to \aut(H)$ defined as $f\mapsto \iota_f$ defines an homomorphism. (Here $\iota_f(h)=fhf^{-1}~\forall h\in H$.)    This is a monomorphism since $\iota_f$ is non-trivial on $S_\infty\subset H$ if $f$ is not 
the identity.  

Let $\phi:H\to H$ be any automorphism and let $f\in S_\omega$ be the element such that $\phi|_{S_\infty}=\iota_f$.
We claim that $\phi=\iota_f$.   
Suppose that $u:=\phi(h), \iota_f(h)=fhf^{-1}=:v$ for some $h\in H$.  We must show that $u(i)=v(i)$ for all $i\in \mathbb{N}$.  It suffices 
to show that $\{u(i), u(j)\}=\{v(i),v(j)\}$ for all $i, j\in \mathbb{N}, i\ne j$.
Let $i,j\in \mathbb{N}$, $i\ne j$. 
Now consider the transposition $(a,b)\in S_\infty$ such that 
$\iota_f(a,b)=\phi(a,b)=(i,j)$.  We have 
$\phi(h(a,b)h^{-1})=\phi(h)\phi(a,b)\phi(h^{-1})=u(i,j)u^{-1}=(u(i),u(j))$, while $\iota_f(h(a,b)h^{-1})=
\iota_f(h) \iota_f(a,b)\iota_f(h^{-1})=v(i,j)v^{-1}=(v(i),v(j))$.   Therefore $(u(i), u(j))=(v(i),v(j))\in S_\infty$ since 
$\iota_f$ and $\phi$ agree on $S_\infty$.  This implies that $\{u(i), u(j)\}=\{v(i),v(j)\}$, completing the proof. 
\end{proof}

\subsection{$S_\infty$ has the $R_\infty$-property}
Let $\theta\in \aut(S_\infty)$.   In view of Lemma \ref{autsym}
$\theta=\iota_f$ for some $f\in S_\omega$.  Let $x, y\in S_\infty$ and suppose that $y=zx\theta(z^{-1})=zxfz^{-1}f^{-1}$ for some $z\in S_\infty$.   Then $yf=z(xf)z^{-1}$ in $S_\omega$ for some $z\in S_\infty$.    
For any cycle (finite or infinite) $u=(u_j)$, we have that $zuz^{-1}$ is the cycle $(z(u_j)).$    Any $z\in S_\infty$  
moves only finitely many elements of $\mathbb{N}$.  Hence  when $u$ is an infinite cycle $z(u_j)=u_j$ for all but 
finitely many $j\in \mathbb{Z}$.  For an arbitrary element $u$ expressed as a product of pairwise disjoint cycles, $u(\alpha)=(u(\alpha)_j)$, 
the element $zuz^{-1}$ being a product of $zu(\alpha)z^{-1}$,  we see that $zu(\alpha)z^{-1}=u(\alpha)$ for all but 
a finitely many $\alpha$, and, moreover, if $u(\alpha)=(u(\alpha)_j)_{j\in \mathbb{Z}})$ is an 
infinite cycle, then $z(u(\alpha)_j)=u(\alpha)_j$ for all but finitely many $j\in \mathbb{Z}$.
\footnote{There is a mild abuse of notation here; $u(\alpha)$ is not to be confused with the value of $u$ at $\alpha$. 
We will use Greek letters as labels in such situations.} 

\begin{lemma} \label{infcycle}
Suppose that $f\in S_\omega$ has an infinite cycle $u$, then there exist infinitely many transpositions $\tau_k\in S_\infty$ 
such that $\tau_j f \ne z\tau_k fz^{-1}$  for any $z\in S_\infty$.
\end{lemma}
\begin{proof}
Fix an infinite cycle $u=(u_\alpha)_{\alpha\in \mathbb{Z}}$ that occurs in the cycle decomposition of $f$.  Let $\tau_\alpha=(u_0, u_\alpha), \alpha\ge 1$.  
Then we claim that $\tau_\alpha f$ and $\tau_\beta f$ are not conjugates if $\alpha\ne \beta$.  To see this, we apply Lemma \ref{transcycle} to 
compute $\tau_\alpha u, \alpha\ge 1.$   Note that the cycles that occur in $\tau_\alpha u$ also occur in the cycle decomposition 
of $\tau_\alpha f$.  This is true in particular of the infinite cycle, denoted $v(\alpha)$, that occurs in $\tau_\alpha u$.

  Now 
 $v(\alpha)_p=v(\beta)_p=u_p$ for all $p<0$ and $\alpha,\beta\ge 1$, 
and, when $\alpha\ne \beta$, $u_{p+\alpha}=v(\alpha)_p\ne v(\beta)_p=u_{p+\beta}, p\ge 0$.  
This implies that the $zv(\beta)z^{-1}$ cannot occur in $\tau_\alpha f$  for any $z\in S_\infty$ if $\alpha \ne \beta$ in its  
cycle decomposition by the assertion made in the paragraph above the statement of the lemma.   Hence $\tau_\alpha f\ne z\tau_\beta fz^{-1}$ 
for any $z\in S_\infty$.   \end{proof}
 
We are now ready to prove the first assertion of Theorem \ref{main}.

\begin{theorem}\label{symmetricgroup} 
The group $S_\infty$ has the $R_\infty$-property. 

\end{theorem}
\begin{proof}  
Let $\theta=\iota_f\in\aut(S_\infty)$ where $f\in S_\omega$.   
We need to show that there exists pairwise distinct elements $\tau_j\in 
S_\infty, j\in \mathbb{N}$ such that $\tau_jf\ne z\tau_kfz^{-1}$ for any $z\in S_\infty$ if $j\ne k$.  
Since $S_\infty$ has infinitely many conjugacy classes, 
the assertion holds for $f\in S_\infty$, we need only consider the case $f\notin S_\infty$.  Thus, 
in the cycle decomposition of $f$, 
either (i) there exist an infinite cycle, or, (ii)  all the cycles are finite and there are infinitely many 
of them.  
 
Case (i). In this case the assertion has already been established in Lemma \ref{infcycle}.  

Case(ii). Suppose that $f=\prod_{\alpha\in \mathbb{N}}u(\alpha)$ where the 
$u(\alpha)$ are all finite cycles having length $\ell(\alpha)$ at least $2$ for every $\alpha\in \mathbb{N}$.   
Let $J:=\{\alpha\in \mathbb{N}\mid \ell(\alpha)\ge 3\}$. 
We break up the proof into two subcases depending on whether $J$ is infinite or not.

Subcase (a):  $J$ is infinite.   Let  $J_k\subset J$ be the set consisting of the first $k$ elements 
of $J$ (with respect to the usual ordering on $J\subset \mathbb{N}$). 
Write $u(\alpha)=(u(\alpha)_1,\ldots, u(\alpha)_{\ell(\alpha)})$ and set $U(\alpha):=\{u(\alpha)_i\mid 1\le i\le \ell(\alpha)\}, \alpha\in \mathbb{N}$.  
Consider the collection of pairwise disjoint transpositions $\lambda_\alpha=(u(\alpha)_1, u(\alpha)_2), \alpha\in J,$ and let $\tau_k=\prod_{\alpha\in J_k}\lambda_\alpha$. Note that $\lambda_\alpha u(\alpha)=(u(\alpha)_1)
.(u(\alpha)_2,\ldots, u(\alpha)_{\ell(\alpha)})=(u(\alpha)_2,\ldots, u(\alpha)_{\ell(\alpha)})$ fixes only $u(\alpha)_1$ in the set $U(\alpha)$ as $\ell(\alpha)\ge 3$.   
 Then $\tau_k.\prod_{\alpha\in J_k}u(\alpha)$ fixes only $u(\alpha)_1\in \mathbb{N}, \alpha\in J_k,$ in the set $\cup_{\alpha\in J_k}U(\alpha)$.    Let $F_0=\fix(f)$.  Then $\fix(\tau_k f)=F_0\cup \{u(\alpha)_1\mid \alpha\in J_k\}=:F_k$.  

Suppose that $\tau_jf=z\tau_k f z^{-1}$ with $z\in S_\infty$ with $j\ne k$.  Then $z$ defines a bijection   
$\zeta: F_j\to F_k$ between the fixed sets of $\tau_jf$ and $\tau_k f$.    
Clearly this is a contradiction if $\fix(f)=F_0$ is finite. 
Assume that $F_0\subset \mathbb{N}$ is infinite.  Since $z\in S_\infty$, it fixes all but finitely many 
elements of $F_0$.    Let $L:=\{m\in F_0\mid z(m)\ne m\}$.    Note that $\zeta$ restricts to the identity 
on $F_0\setminus L$.  Therefore $\zeta$ restricts to a 
bijection between $L\cup \{u(\beta)_1\mid \beta\in J_j\}$ and $L\cup \{u(\beta)_1\mid \beta\in J_k\}$.  Since $j\ne k$,  
$L$ is finite and $L\subset F_0$ is disjoint from  $\{u(\beta)_1\mid \beta\in J_n\}, n=j,k$, this 
is a contradiction.

Subcase (b):  The set $J$ is finite; we set $K=\mathbb{N}\setminus J$ and define $K_j, j\in \mathbb{N},$ to be the set of 
first $\alpha$ elements of $K$.  
Again we set $\lambda_\alpha=(u(\alpha)_1), u(\alpha)_2)=u(\alpha), \alpha\in K$.    Now, if $\alpha\in K$, we have $\lambda_\alpha u(\alpha)=id$, that is, $\lambda_\alpha u(j)$ fixes both points of $U(\alpha)$.  We set 
$\tau_j:=\prod_{\alpha \in K_j} \lambda_\alpha, F_j:=\fix(\tau_jf)=F_0\cup_{\alpha \in K_j}U(\alpha)$.   
Arguing exactly as above we see that for any $z\in S_\infty$, $\tau_jf=z\tau_kfz^{-1}$ implies $j=k$, completing the proof.
\end{proof}


\subsection{Houghton groups}
As in the introduction, $H_n, n\ge 2,$ denotes the Houghton group.   
We first describe the group of outer automorphisms of $H_n$.  Recall from \S1 that one has an exact sequence 
\[1\to S_\infty(M_n)\hookrightarrow H_n\stackrel{\tau}{\to} Z\to 1\]
where 
$\tau:H_n\to Z$ send $f\in H_n$ to the translation part $(t_1,\ldots, t_n)\in Z$ of $f$.   The group $S_\infty(M_n)$ 
is the commutator subgroup of $H_n$ if $n\ge 3$. When $n=2$, the commutator subgroup is the alternating group 
$A_\infty(M_2)$ which has index $2$ in $S_\infty(M_2)$.  In any  case, $S_\infty=S_\infty(M)$ is characteristic in $H_n$ 
as $H_n/S_\infty$ is the maximal {\it torsion-free} abelian quotient of $H_n$.

\begin{lemma}\label{outh}
Let $\phi:H_n\to H_n, n\ge 2,$ be an automorphism.  
Then $\phi$ is inner if and only if $\bar{\phi}:Z\to Z$ is the identity automorphism. 
\end{lemma}
\begin{proof}
It is trivial to see that any inner automorphism of $H_n$ induces the identity automorphism of $Z$.  
For the converse, 
suppose that $\phi:H_n\to H_n$ induces the identity automorphism of $Z$.  

Let $f\in S(M_n)$ be such that $\iota_f(H_n)=H_n$.

Consider the element $h_{p}:M_n\to M_n,1\le  p<n,$ in $H_n$ defined follows: 
\footnote{The element $(p,k)\in M_n$ should not be confused with the transposition in $S(\mathbb{N})$.}
\[h_p(i, k)=\left \{ 
\begin{array}{cc} (p,k+1),& \text{if}~ i=p, k\ge 1\\
       (n,k-1) ,& \textrm{if}~ i=n, k>1, \\ 
       (p,1), &\textrm{if} ~i=n,k=1\\
       (i,k), & \textrm{ if} ~ i\ne p,n.\\
       \end{array}
       \right. 
       \]
Thus  $h_p$ permutes $\{p,n\}\times \mathbb{N}$ in a single cycle:
\[
h_p=(\ldots, (n,2), (n,1), (p,1), (p,2), \ldots, (p,k),\ldots)\]
and so $fh_pf^{-1}$ is the cycle \[fh_pf^{-1}
=(\ldots, f(n,2), f(n,1), f(p,1), f(p,2), \ldots, f(p,k), \ldots) \in H_n.\]   
The only infinite cycles in $H_n$ are those whose terms, except for a finite part of the cycle, 
are {\it consecutive} numbers along two rays, say $\{i_n\}\times \mathbb{N}$ 
 and $\{i_p\}\times \mathbb{N}$, in the negative and positive directions respectively of the cycle $fh_pf^{-1}$. 
 Therefore we have  
$\tau(fh_pf^{-1})=e_{i_p}-e_{i_n}.$   Moreover, there exist integers $t_n, t_p$ such that 
$f(n,k)=(i_n, k+t_n), f(p, k)=(i_p, k+t_p)$ for sufficiently large $k$.   Clearly $i_n$ and $t_n$ are independent of $p$. 
Since $f$ is a bijection, the 
association 
$p\mapsto i_p$ is a permutation $\pi_f\in S_n$, and 
$\sum_{1\le q\le n} t_q=0$.   Note that $\pi_f=id$ if and only if $f\in H_n$.

Since $S_\infty$ is characteristic in $H_n$, 
by Corollary \ref{autgen}, $\phi=\iota_g$ for a unique $g\in S(M_n)$.   
{\it We claim that $g\in H_n$.}  
Since $\tau(ghg^{-1})=\tau(\phi(h)) =\tau(h)~\forall h\in H_n$, we have $\pi_g(q)=q$ for all $q\le n$ 
and so we have $g\in H_n$.    
\end{proof}

The group $S_n$ acts on the set $M_n=\{1,\ldots, n\}\times \mathbb{N}$ in the obvious manner, by acting via the identity 
on $\mathbb{N}$. This defines an action $\psi$ of $S_n$ on the group 
$S(M_n)$ defined as $f\mapsto \sigma \circ f\circ \sigma^{-1}$ which preserves the subgroup $H_n$.  Thus we obtain 
a homomorphism $\psi:S_n\to \aut(H_n)$.  It is readily seen that $\tau(\psi_\sigma(h))=\sigma(\tau(h)) ~\forall h\in H_n,$ where 
$\sigma$ acts on $Z\subset \mathbb{Z}^n$ by permuting the standard basis elements $e_1,\ldots, e_n$.  In particular 
$\psi$ is a monomorphism.  Let $\bar{\psi}:S_n \to \out(H_n)$ be the composition of $\psi$ with the projection 
$\aut(H_n)\to \out(H_n)$.              

\begin{proposition} \label{outsym}
The homomorphism $\bar{\psi}: S_n \to \out(H_n)$ is an isomorphism and so $\aut(H_n)=Inn(H_n)\rtimes S_n
\cong H_n\rtimes S_n$.
\end{proposition}
\begin{proof}
Lemma \ref{outh} shows that $\bar{\psi}$ is a monomorphism.  We shall show that it is surjective.  

Let $\phi\in\aut(H_n)$.  Write $\phi=\iota_f$ for a (unique) $f\in S(M_n)$.  With notations as in the proof of 
Lemma \ref{outh}, 
let $\pi:=\pi_f\in S_n$.

Consider the automorphism $\psi_\pi^{-1}\phi=:\theta$.
We have $\tau(\theta(h_p))=\pi^{-1}(\tau(\phi(h_p)))=\pi^{-1}(\tau(fh_pf^{-1}))=\pi^{-1}(e_{\pi(p)}-e_{\pi(n)}) 
=e_p-e_n=\tau(h_p), 1\le p<n$.     
Since the group $Z$ is generated by $\tau(h_p), 1\le p<n,$  
it follows by Lemma \ref{outh} that $\theta$ is inner.   Hence $\bar{\psi}(\pi)=\phi \mod Inn(H_n).$

Finally note that $Inn(H_n)\cong H_n$ since the centre of $H_n$ is trivial. 
\end{proof}

The above description of $\aut(H_n)$ had been obtained by Burillo, Cleary, Martino, and R\"over \cite{bcmr} 
and also by Cox \cite{cox}.    
Our 
proof is essentially based on Corollary \ref{autgen}, which is applicable in a more general context.

\begin{theorem}\label{houghton}
The Houghton group $H_n$ has the $R_\infty$-property for any $n\ge 2$.
\end{theorem}

We shall give two proofs.  The first one uses the structure of the 
automorphism group of $H_n$ and is more direct. The second one uses the result of 
Theorem \ref{symmetricgroup} and the addition formula (Lemma \ref{additionformula}).

{\it First proof.}   
First observe that there are infinitely many conjugacy classes in $H_n$ since two elements 
in $S_\infty =S_\infty(M_n)\subset H_n$ are conjugates in $H_n$ only if they have the same cycle type.  It follows that 
$R(\phi)=\infty$ for any inner automorphism $\phi$ of $H_n$.    Therefore, to show that $R(\phi)=\infty$ for an 
{\it arbitrary} $\phi\in \aut(H_n)$,  it suffices 
to show that $R(\phi)=\infty$ for all $\phi$ in a set of coset representatives of elements of $\out(H_n)$.  
Thus we need only show that $R(\psi_\sigma)=\infty$ for any $\sigma\in S_n$, where $\psi:S_n\to \aut(H_n)$ 
is as defined in the paragraph above Proposition \ref{outsym}.  We shall use Lemma \ref{finiteorder} and Remark \ref{torsionfix} to achieve this.

For $k\ge 1$, consider the element $\xi_k$ defined as the product of $k$-cycles $((i, 1),\ldots, (i,k))\in H_n$, $1\le i\le n$.  Explicitly, 
  \[\xi_k(i, j)=\left\{ \begin{array}{cc} 
                             (i,j+1) & \textrm{if} ~ 1\le j<k,\\
                             (i,1)& \textrm{if}~ j=k,\\
                             (i,j)& \textrm{if} ~j>k,
                             \end{array}
                             \right.
                             \]                             
for all $i\le n$.  Then $\xi_k$ is fixed by $\psi_\sigma $ for every $\sigma\in S_n$.  Thus, $\{\xi_{k}^n\mid k\ge 1\}$ 
contains elements of arbitrarily large orders and so by Remark \ref{torsionfix} it follows that $R(\psi_\sigma)=\infty$ 
for all $\sigma\in S_n$, completing the proof.

{\it Second proof.}    Consider the exact sequence $1\to S_\infty(M_n)\to H_n \to Z\to 0$.  
As remarked already, $S_\infty(M_n)$ is characteristic in $H_n$ and we have $Z\cong\mathbb{Z}^{n-1}$. 
Thus any automorphism $\theta$ of $H_n$ restricts to an automorphism $\theta'$ of $S_\infty(M_n)$ and 
induces an automorphism $\bar{\theta}$ of $Z$.  If $R(\theta')=\infty$ then, by
Lemma \ref{additionformula} (i), we have $R(\theta)=R(\bar{\theta})=\infty$.  
Now suppose that $R(\bar{\theta})<\infty$. Then $Fix(\bar{\theta})=0$.   Since $Z$ is abelian and 
since $R(\theta')=\infty$ by Theorem \ref{symmetricgroup}, the addition formula  (Lemma \ref{additionformula}(ii)) yields $R(\theta)=R(\theta')=\infty$, completing the proof.  \hfill $\Box$

\subsection{The group of pure symmetric automorphisms} \label{gpsa}
Recall that $G_n\subset \aut(F_n), n\ge 2,$ denotes the group of pure symmetric automorphisms 
of the free group $F_n$ of rank $n$.  A presentation for $G_n$, obtained by McCool, was recalled 
in \S1.  It is immediate from this presentation that the abelianization $G_n^\ab=G_n/[G_n,G_n]$ is isomorphic 
to $\mathbb{Z}^{n^2-n}$ with basis the images $\bar{\alpha}_{ij}, 1\le i\ne j\le n$.  We denote by 
$\{\chi_{ij}\mid 1\le i\ne j\le n\}$ the basis of $\hom(G_n^\ab,\mathbb{Z})$ dual to the basis $\{\bar{\alpha}_{ij}\mid 1\le i\ne j\le n\}$. 
We shall denote by the same symbol $\chi_{ij}$ the composition $G_n\to G_n^\ab\stackrel{\chi_{ij}}{\to }\mathbb{Z}\hookrightarrow \mathbb{R}$.   
We will assume that $n\ge 3$, leaving out $G_2$ which is isomorphic to a free group of rank $2$, which is known to have the $R_\infty$-property. 

We begin by recalling the explicit description 
of $\Sigma^c(G_n)$ due to Orlandi-Korner \cite{o}.  

Let $A_{ij}:=\mathbb{R}\chi_{ij}+\mathbb{R}\chi_{ji}, 
B_{ijk}:=\mathbb{R}(\chi_{ij}-\chi_{kj})+\mathbb{R}(\chi_{jk}-\chi_{ik})+\mathbb{R}(\chi_{ki}-\chi_{ji})$, $i,j,k$ pairwise  
distinct.    Note that $A_{ij}=A_{ji}$ and $B_{ijk}=B_{pqr}$ if $\{i,j,k\}=\{p,q,r\}$.  
Let $S$ be union of vector subspaces $S=\bigcup A_{pq}\cup \bigcup B_{ijk}\subset \hom(G_n,\mathbb{R})$ where the unions are over all pairs of distinct number $p,q\le n$ and 
all pairwise distinct numbers $i,j,k\le n$. It was shown by Orlandi-Korner \cite{o} that $\Sigma^c(G_n)$ is the image of $S\setminus\{0\}\subset \hom(G_n,\mathbb{R})\setminus \{0\}$.  

Let $\mathcal{S}_n$ denote the semidirect product $C_2^n\rtimes S_n$ where $S_n$ acts 
on $C_2^n$ by permuting the coordinates.  Here $C_2=\{1,-1\}$.
The group $\mathcal{S}_n$ acts effectively on $F_n$ the free group with basis $\{x_1,\ldots, x_n\}$ where 
$\pi\in S_n$ permutes the generators: $\pi(x_j)=x_{\pi(j)}, 1\le j\le n$ and the action of the $k$-th factor 
of $C_2^n$ is given by the automorphism $t_k(x_k)=x_k^{-1}, t_k(x_j)=x_j, j\ne k.$
Thus $\mathcal{S}_n$ is a subgroup of $\aut(F_n)$.  It is readily verified that $\mathcal{S}_n$ normalizes 
$G_n$:  $t_k\alpha_{i,j}t_k^{-1}= \alpha_{i,j}^{-1}$ if $k=j$ and equals $\alpha_{i,j}$ otherwise; 
if $\pi\in S_n$, then $\pi\alpha_{i,j}\pi^{-1}=\alpha_{\pi(i),\pi(j)}$ for all $i,j$.   
In particular $\pi^*(A_{ij})=A_{\pi(i)\pi(j)}, \pi^*(B_{ijk})=B_{\pi(i)\pi(j)\pi(k)}$ for all $\pi\in S_n$.
Thus we have the 
following

\begin{lemma} \label{signsym}
Let $n\ge 3$.
The action of the group $\mathcal{S}_n\subset \aut(G_n)$ on $\hom(G_n,\mathbb{R})$ and on $\Sigma^c(G_n)$ 
is defined by $ \pi^*(\chi_{i,j})=\chi_{\pi(i),\pi(j)}, t^*(\chi_{i,j})=t_it_j\chi_{i,j},$  
for all $\pi\in S_n, t=(t_1,\ldots, t_n)\in C_2^n$.   \hfill $\Box$
\end{lemma}

The following proposition is refinement of a statement of Gon\c{c}alves and Kochloukova in proof of 
\cite[Theorem 4.11]{gk}). 

\begin{proposition} \label{gsym}
There exists a homomorphism $\eta:\aut(G_n) \to S_n$ which is surjective such that $\phi^*(\chi_{i,j})=\epsilon_{i,j}
\chi_{\sigma(i),\sigma(j)}, 1\le i\ne j\le n,$ where $\epsilon_{i,j}\in \{1,-1\}$ and $\sigma=\eta(\phi)\in S_n$. 
In particular, $\aut(G_n)\cong K\rtimes S_n$ where $K=\ker(\eta)$.
\end{proposition}    
\begin{proof}
Since $\phi^*$ is a linear isomorphism of  $\hom(G_n,\mathbb{R})$ and since 
$\phi^*:\Sigma^c(G_n)\to \Sigma^c(G_n)$ 
is a homeomorphism, $\phi^*$ preserves the collections of subspaces 
$\mathcal{A}:=\{A_{ij}\mid 1\le i<j\le n\}$ 
and $\mathcal{B}:=\{B_{ijk}\mid 1< i<j<k\le n\}$.  Note that $\mathcal{B}$ is non-empty since $n\ge 3$.
{\it In our notation $A_{pq}, B_{pqr}$ it is not assumed that $p<q<r$.}

It is readily seen that $(A_{pq}+A_{rs})\cap B_{ijk}=0$ unless $\{p,q,r,s\}=\{i,j,k\}$.
On the other hand $(A_{ij}+A_{ik})\cap B_{ijk}=\mathbb{R}(\chi_{k,i}-\chi_{j,i})$.  It follows that $\phi^*$ preserves 
the collection of $1$-dimensional spaces $\mathcal{C}:=
\{\mathbb{R}(\chi_{k,i}-\chi_{j,i})\mid \textrm{$i,j,k$ pairwise distinct}\}$.  

Let $\phi^*(A_{ij})=A_{pq}$,  
$\phi^*(A_{ik})=A_{rs},$ where $i,j,k$ are pairwise distinct.  Then $\{p,q\}\cap\{r,s\}$ is a singleton, say $s=p$---so that $\phi^*(A_{ik})=A_{pr}$---and 
$\phi^*(B_{ijk})=B_{pqr}.$  For, otherwise $(A_{ij}+A_{i,k})\cap B_{ijk}$ is one-dimensional whereas 
$\phi^*((A_{ij}+A_{ik})\cap B_{ijk})=(A_{pq}+A_{pr})\cap \phi^*(B_{ijk})=0$.

In view of the fact that $\phi^*$ stabilizes $\mathcal{C}$, we have 
\[\phi^*(\chi_{k,i}-\chi_{j,i})=a(\chi_{r,p}-\chi_{q,p}).\eqno(*)\] 
On the other hand $\chi_{k,i}\in A_{ik}$ and so $\phi^*(\chi_{k,i})\in \phi^*(A_{ik})=A_{pr}$ and so
$\phi^*(\chi_{k,i})=b\chi_{p,r}+c\chi_{r,p}$ for some $b,c\in \mathbb{R}$; similarly $\phi^*(\chi_{j,i})=b'\chi_{q,p}+c'\chi_{p,q}$ for some $b',c'\in \mathbb{R}$. Therefore 
\[\phi^*(\chi_{k,i}-\chi_{j,i}) =b\chi_{p,r}+c\chi_{r,p}-b'\chi_{q,p}+c'\chi_{p,q}.\eqno(**)\]

Comparing ($*$) and ($**$) we see that $b=0=c'$, that is, $\phi^*(\chi_{k,i})=c\chi_{r,p}$ and $\phi^*(\chi_{j,i})=b'\chi_{q,p}$.  Since $\phi^*:\hom(G_n;\mathbb{R})\to \hom(G_n,\mathbb{R})$ preserves the lattice $\hom(G_n,\mathbb{Z})$ and 
since  
$\chi_{k,i},\chi_{j,i}$ are part of a $\mathbb{Z}$-basis of $\hom(G_n,\mathbb{Z})$, we see that $c,b'=\pm 1$. 

To complete the proof, we define the permutation $\sigma\in S_n$ associated to $\phi\in \aut(G_n)$ as $\sigma(i)=p$, (with notation as above).  Note that $\sigma$ is indeed a bijection since $\phi^*$ is an isomorphism.  We define 
$\eta:\aut(G_n)\to S_n$ by $\eta(\phi)=\sigma$.  Then $\eta$ is a homomorphism of groups.  It is surjective 
since its restriction to $S_n\subset \mathcal{S}_n$ is the identity by Lemma \ref{signsym}.  This also shows that 
$\eta$ splits, completing the proof.
\end{proof}

\begin{remark}
It seems plausible that there exists a surjective homomorphism $\tau:\aut(G_n)\to \mathcal{S}_n$ such that 
$\phi^*(\chi_{i,j})=t_it_j\chi_{\sigma(i),\sigma(j)}, 1\le i\ne j\le n,$ where $\tau(\phi)=(t_1,\ldots,t_n)\in 
C_2^n, \sigma=\eta(\phi) \in S_n.$   This would then imply that 
$\aut(G_n)\cong N\rtimes \mathcal{S}_n$ for a suitable subgroup $N\subset \aut(G_n)$.
\end{remark}

The above proposition says that the matrix of $\phi^*$ with respect to the basis $\{\chi_{i,j}\mid 1\le i\ne j\} $ (ordered by, say, the lexicographic ordering of the indices $i,j$), is of the form $\phi^*=DP$ where $D$ is a diagonal matrix with 
eigenvalues $\pm 1$ and $P$ is a permutation matrix.  

\begin{lemma}  
Let $T=DP$ where $D,P\in M_m(\mathbb{R})$ are such that $D$ is a diagonal matrix and $P$ is a permutation 
matrix.  If $P=P_1.\ldots .P_k$ is a cycle decomposition then there exist eigenvectors $v_1,\ldots, v_k$  which 
are linearly independent. 
\end{lemma}
\begin{proof}   The cycle decomposition allows us to express $\mathbb{R}^n$ as a direct sum $V_1\oplus\cdots V_k$ 
where $V_j$ is spanned by $\{e_i\mid P_j(i)\ne i\}.$  Specifically, if $P_j=(i_1, \ldots, i_k)$.  Then 
$v_j:=e_{i_1}+d_{i_1}e_{i_2}+\ldots+ d_{i_1}\ldots d_{i_{k-1}}e_{i_k}$  which is the sum of the vectors in the 
$DP$-orbit of $e_{i_1}$, is an eigenvector of $T$ with eignevalue 
$d_{i_1}\ldots d_{i_k}$.  Evidently $v_1,\ldots, v_k$ are linearly independent.
\end{proof}

We will use the above lemma to construct two linearly independent eigenvectors of $\phi^*$ (with further properties
that will be relevant for our purposes).
Let $\sigma=\eta(\phi)\ne id$ and $\phi^*=DP$ with $D$ diagonal and $P$ a permutation transformation (with respect to the 
basis $\{\chi_{i,j}\}$).   Suppose that $\sigma$ has a $k$-cycle in its cycle decomposition where $k>2$. 
Choose any $i$ that occurs in it and let $j:=\sigma(i)$.  Then $\chi_{i,j}$ and $\chi_{j,i}$ do not occur in the 
same orbit of $DP$ and so $v_{i,j}:=\sum_{0\le r<k} (DP)^r(\chi_{i,j})$ and $v_{j,i}:=\sum_{0\le j<k}
 (DP)^r\chi_{j,i}$ are eigenvectors of the same eigenvalue $\epsilon\in \{1,-1\}.$  Without loss of generality 
 we assume that $i=1, j=2$ and set $v_{1,2}=:u, v_{2,1}=v$. 
Suppose that there is no such $k$-cycle in $\sigma$. Then $\sigma$ is a product of disjoint transpositions. 
Without loss of generality, suppose that the transposition 
$(1,3)$ occurs in the decomposition.  Since $n>2$, either $\sigma$ has a 
fixed point, say $2$, or $n>3$ and, say the transposition $(2,4)$ occurs in the decomposition.
In the first case $u:=\chi_{1,2}+d_{1,2}\chi_{3,2}, v:=\chi_{2,1}+d_{2,1}\chi_{2,3}$ are eigenvectors of $P$ 
and in the latter case, $u:=\chi_{1,2}+d_{1,2}\chi_{3,4}$ and $v:=\chi_{2,1}+d_{2,1}\chi_{4,3}$ are eigenvectors 
of $P$.   Thus in {\it all} case $\chi_{1,2}$ occurs in $u$ and $\chi_{2,1}$ occurs in $v$ where $u,v$ are 
eigenvectors of $\phi^*$.  If $1$ is an eigenvalue of $\phi^*$, then $\bar{\phi}$ has a non-zero fixed element 
so $R(\phi)=\infty$. So assume that $\phi^*(u)=-u, \phi^*(v)=-v$.  Then there exists elements $\beta,\gamma\in G_n$ 
such that $\bar{\phi}(\bar{\beta})=-\bar{\beta}, \bar{\phi}(\bar{\gamma})=-\bar{\gamma}$ where $\bar{\alpha}_{1,2}, \bar{\alpha}_{2,1}$ occur in $\bar{\beta},\bar{\gamma}$ respectively, 
with coefficient $1$.  
  
Denote by $\Gamma_2:=\Gamma_2(G_n)$ the commutator subgroup of $G_n$ and by $\Gamma_3:=\Gamma_3(G_n)$ the subgroup $[G_n, \Gamma_2]
\subset \Gamma_2$.  Thus $G_n/\Gamma_3$ is a two-step-nilpotent group and 
we have the following exact sequences:\\
\[ 1 \to \Gamma_3\to G_n\to G_n/\Gamma_3 \to 1,\]
\[1\to \Gamma_2/\Gamma_3\to G_n/\Gamma_3\to G_n/\Gamma_2\to 1\]
 
Since $\Gamma_2$ and  
$\Gamma_3$ are characteristic in $G_n$ any automorphism of $G_n$ restricts to automorphisms 
of $\Gamma_2$ and $\Gamma_3$ and hence induces automorphisms of the quotients $G/\Gamma_3, \Gamma_2/\Gamma_3$ and $G_n/\Gamma_2=G_n^\ab$.   

Denote by $\theta\in \aut(G_n/\Gamma_3)$ the automorphism defined by $\phi$ and $\theta'$, the restriction of $\theta$ to $\Gamma_2/\Gamma_3$.  With notations as above, 
$[\beta,\gamma]\Gamma_3\in \Gamma_2/\Gamma_3$ satisfies $\theta'([\beta,\gamma]\Gamma_3)=[\beta,\gamma]\Gamma_3$.
By using the addition formula (Lemma \ref{additionformula}),  we conclude that $R(\theta)=\infty$, {\it provided} 
$[\beta,\gamma]/\Gamma_3$ is of infinite order.  Granting this for the moment, by the first part of the same lemma we 
conclude that $R(\phi)=\infty$ using the first exact sequence above. Since $\phi\in \aut(G_n)$ was arbitrary, we conclude that $G_n$ has the $R_\infty$-property.  
So all that remains is to show that $[\beta,\gamma]\Gamma_3$ is not a torsion element.

We use the fact that under the surjection $\psi: G_n\to G_2$ that maps $\alpha_{i,j}$ to $\alpha_{i,j}$ when $\{i,j\}=\{1,2\}$ 
and the remaining $\alpha_{i,j}$ to $1$,  
$\Gamma_k$ maps onto $\Gamma_k(G_2)$, k=2,3.    Let $\beta_2,\gamma_2\in G_2$ be the images of $\beta, \gamma$ respectively under $\psi$.  Then $\bar{\beta}_2=\bar{\alpha}_{1,2},\bar{\gamma_2}=\bar{\alpha}_{2,1}\in G_2^\ab$.
Therefore $[\beta_2,\gamma_2]\Gamma_3(G_2)=[\alpha_{1,2},\alpha_{2,1}]\Gamma_3(G_2)$.  Since $G_2$ is a free 
group with basis $\{\alpha_{1,2},\alpha_{2,1}\}$ we see that $[\alpha_{1,2},\alpha_{2,1}]\Gamma_3(G_2)$ generates 
an infinite cyclic group.  Hence the same is true of $[\beta,\gamma]\Gamma_3$.
This completes the proof of part (iii) of the main theorem, which is restated below:

\begin{theorem}\label{pure}
The group $G_n, n\ge 2,$ has the $R_\infty$-property. \hfill $\Box$
\end{theorem}

\section{The Thompson group $T$}

Recall, from \S1, the description of the Richard Thompson group $T$ as the group of 
all orientation preserving piecewise linear homeomorphisms of $\mathbb{S}=I/\{0,1\}$ 
with slopes in the multiplicative group generated by $2\in \mathbb{R}_{>0}$ and break 
points in $\mathbb{Z}[1/2]$.    We regard the Thompson group $F$ as the subgroup 
of $T$ consisting of elements which fix the element $1\in \mathbb{S}^1$. 
In this section we prove the following result.  

\begin{theorem} \label{t} {\em (\cite{bmv}, \cite{gs}).}
The Richard Thompson group $T$ has the $R_\infty$-property.
\end{theorem}

 The fact that $T$ has the $R_{\infty}$ property has been proved first  by Burillo, Matucci, and Ventura \cite{bmv} (see also \cite{gs}).  The crucial
point in the proofs of the result above is the same  in both \cite{bmv} and \cite{gs} and both the proofs rely on the 
description of the outer automorphism of $T$ (recalled in Theorem \ref{brin} below). 
However, since the approaches before getting to the main point are slightly different, we provide our proof here which may contain some features that are useful for other situations (such as in Remark \ref{gent} below).

 It is readily seen that 
the reflection map $r$ defined as $r(x)=1-x, x\in [0,1]$, 
induces an automorphism $\rho: T\to T$ defined as $\rho(f)=r\circ f\circ r^{-1}=r\circ f\circ r$. We now state 
the following result of Brin.

\begin{theorem} \label{brin} {\em (Brin \cite{brin})} 
 The group of inner automorphisms of $T$ is of index two in $\aut(T)$ and the 
quotient group $\out(T)$ is generated by 
$\rho$. 
\end{theorem}

As observed in \S\ref{periodic},  for any group $\Gamma$ and any automorphism $\phi\in \aut(\Gamma)$, and any $g\in \Gamma$, 
$R(\phi)=\infty$ if and only if $R(\phi\circ\iota_g)=\infty$.  Therefore, to establish the $R_\infty$-property for $\Gamma$, 
it is enough to show that $R(\phi)=\infty$ for a set of coset representatives of $\out(\Gamma)$.
  In the case 
$\Gamma=T$, in view of  Theorem \ref{brin} due to Brin,  we need only show that $R(\rho)=\infty$ and $R(id)=\infty$.    
The latter equality is established in 
Proposition \ref{conjugacy} as an easy consequence of Lemma \ref{support} below.  Since $\rho^2=id$, we may 
apply Remark \ref{torsionfix} to show that $R(\rho)=\infty$.   The main idea is to make use of 
homeomorphisms in $Fix(\rho)$, whose supports have arbitrarily large number of disjoint intervals in $\mathbb{S}^1.$ 
(This was also the idea used in the proof by Burillo, Matucci, and Ventura \cite{bmv}.) 

\begin{definition}  {\it Let $X$ be a Hausdorff topological space.\\
(i) The {\em support of}  $f\in \homeo(X)$ is the open set $\supp(f):=\{x\in X\mid f(x)\ne x\}$. \\
(ii) Let $\sigma:\homeo(X) \to \mathbb{N}\cup \{\infty\}$ be  
defined as follows: $\sigma(id)=0$, if $f\ne id$, $\sigma(f)$ is the number of connected components of $\supp(f)$ if it is finite, otherwise $\sigma(f)=\infty$.} 
\end{definition}

\begin{lemma} \label{support} 
Let $\Gamma\subset \homeo(X)$ and let 
$\sigma$ be as defined above.  Suppose that $\theta\in \homeo(X)$ normalizes $\Gamma$.  Then 
$\sigma(f)=\sigma (\theta f \theta^{-1})$.
\end{lemma}
\begin{proof} It is clear that the number of 
connected components of an open set $U\subset X$ remains unchanged under a homeomorphism of $X$.
The lemma follows immediately from the observation 
that $\supp(\theta f \theta^{-1})=\theta(\supp(f))$.  
\end{proof}

\begin{proposition} \label{conjugacy}{\em 
The groups $F$ and $T$ have infinitely many 
conjugacy classes.}
\end{proposition}
\begin{proof}
This follows from Lemma \ref{support} on observing that 
$F$ has elements $f$ with $\sigma(f)$ 
any arbitrary prescribed positive integer.  
Since $F\subset T$, the same is true of $T$ as well.
\end{proof}

\begin{lemma} \label{iterates}
Suppose that $h:\mathbb{R}\to \mathbb{R}$ is an 
orientation preserving homeomorphism.   Then $\supp(h)=\supp(h^k)$ for any non-zero integer $k$. 
\end{lemma}
\begin{proof} Since $\supp(h)=\supp(h^{-1})$ we may 
assume that $k>0$.
Since $h$ is orientation preserving, it is order preserving. 
Suppose that $x\in \supp(h)$ so that $h(x)\ne x$. Say,  $x<h(x)$.  Then applying $h$ to the inequality we obtain $h(x)<h^2(x)$ so that $x<h(x)<h^2(x)$. Repeating this 
argument yields $x<h(x)<\cdots <h^k(x)$ and so $x\in \supp(h^k)$.   The case when $x>h(x)$ is analogous. Thus $\supp(h)\subset \supp(h^k)$.  On the other hand, if $x\notin \supp(h)$, then $h(x)=x$ and so $h^k(x)=x$ for all $k$.  Therefore equality should hold, completing the proof. \end{proof}

We are now ready to prove our main theorem. 

\noindent
{\it Proof of Theorem \ref{t}:}  
By Theorem \ref{brin}(ii), $\out(T)\cong \mathbb{Z}/2\mathbb{Z}$ generated by $\rho$.  By Proposition 
\ref{conjugacy}, $R(id)=\infty$.   
It only remains to verify that $R(\rho)=\infty$.  We apply 
Remark \ref{torsionfix} with $\theta=\rho, n=2, \gamma=1$.  
It remains to show that $\fix(\rho)$ has infinitely many elements $h$ such 
that the $h^2$ are pairwise non-conjugate.  

Let $k\ge 1$.
Let $f_k\in F\subset T$ be such that $\supp(f_k)\subset (0,1/2)$ and has exactly $k$ components. Thus, $\sigma(f_k)=k$. (It is easy to construct such an 
element.) 
Then $\supp(\rho(f_k))
=\supp(rf_kr^{-1})=r(\supp(f_k))\subset (1/2,1)$ is disjoint from $\supp(f_k)\subset (0,1/2)$.  In particular $f_k.\rho(f_k)=\rho(f_k).f_k=:h_k$, $\supp(h_k)=\supp(f_k)\cup r(\supp(f_k))$ and so $\sigma(h_k)=2k$. 
Moreover, since $\rho^2=1$, we see that $h_k\in Fix(\rho)$.
By 
Lemma \ref{iterates}, we have $\sigma(h_k^2)=\sigma(h_k)=2k$.
It follows that $h_k^2$ are pairwise non-conjugate in $T$, completing the proof. \hfill $\Box$.

\begin{remark} \label{gent}
 In the case of the generalized Thompson groups $T_{n,r}$, suppose that $\theta\in \aut(T_{n,r})$ 
is a torsion element, say of order $m$, 
our method of proof of Theorem \ref{t} can be applied to show that $R(\theta)=\infty$.   In fact, applying a theorem of McCleary and Rubin \cite{mr}, to the group $T_{n,r}$, we obtain 
that the automorphism group of $T_{n,r}$ equals its normalizer in the group of all homeomorphisms of the circle $\mathbb{S}^{1}=[0,r]/\{0,r\}$.  Let $\theta \in \aut(T_{n,r})$ and $f\in \mathbb{S}^1$ is such that 
$\theta(x)=fxf^{-1}$ with $f\in \homeo(\mathbb{R}/r\mathbb{Z})$.  
Suppose $f^m=\gamma\in T_{n,r}$ so that $\theta$ represents a torsion element of $\out(T_{n,r})$.  If $\gamma=1$, 
our method of proof of Theorem \ref{t} can be applied to show that $R(\theta)=\infty$.  See \cite{gs} for details.
However, when $\gamma\ne 1$, it is not clear to us how to find elements of $\fix(\theta)$ 
satisfying the hypotheses of Lemma \ref{finiteorder}.  
Our approach yields no information about automorphisms which represent 
non-torsion elements in the outer automorphism group.  The study of the $R_\infty$ property for the 
groups $T_{n,r}$ is a work in progress.

\end{remark} 
\section{Direct product of groups}

It was shown in \cite[Theorem 4.8]{gk} that if $G=G_1\times \cdots \times G_n$ where each $G_i$ is a 
finitely generated group with the property that $\Sigma^c(G_i)$ is a finite set of discrete character classes, not 
all of them empty, then there exists a finite index subgroup $H$ of $\aut(G)$ such that $R(\phi)=\infty$ 
for all $\phi\in H$.  Further, when each $G_i$ is a generalized Richard Thompson 
group $F_{n_i,\infty}, n_i\ge 2,$  then $G$ itself has the $R_\infty$-property.  

We shall strengthen the above result here.   We make use of a result of 
Meinert (as did Gon\c{c}alves and Kochloukova \cite{gk}) that describes the $\Sigma$-invariant of 
a direct product which is recalled below. (Meinert's theorem describes the $\Sigma$-invariant in the more general  
setting of a graph product of groups.)  

Let $G=G_1\times \cdots\times G_n$ and let $r_j=rk(G_j^\ab)$ so that $S(G_j)\cong \mathbb{S}^{r_j-1}$.   We will assume that $r_1\ge 1$.  Then $S(G)=\prod_{1\le j\le n} \hom(G_j,\mathbb{R})\setminus\{0\}/\sim\cong \mathbb{S}^{r-1}$ and so   
$S(G)\cong \mathbb{S}^{r-1}, 
r:=\sum_{1\le j\le n} r_j$.   It is understood that $S(G_j)=\emptyset $ if $r_j=0$. 
The sphere $S(G_i)$ is identified with the subspace of $S(G)$ with 
the set of points with $j$-th coordinate equal to zero for all $j\ne i$.  Observe that $S(G_i)\cap S(G_j)=\emptyset $ 
if $i\ne j$. In order to emphasise this we shall write $S(S_i)\sqcup S(S_j)$ to denote their union, thought of as 
subspaces of $S(G)$.

Recall that $\Sigma^c(G)$ denotes the complement of $\Sigma^1(G)\subset S(G).$

\begin{theorem}  \label{meinert} {\em (Meinert \cite{meinert})}
Let $G=G_1\times \cdots \times G_n$ be finitely generated and let $r_1=rk(G_1^\ab)>0$.  With the above notations,
$\Sigma^c(G)$ equals $\sqcup_{1\le j\le n} \Sigma^c(G_j)$. \hfill $\Box$
\end{theorem}

We will exploit the fact that any $\phi\in \aut(G)$ induces a homeomorphism of the character sphere $S(G)$ which
preserves its rational structure.  Recall that an element $[\chi]\in S(G)$ is called {\it discrete} (or {\it rational}) 
if $Im(\chi)\subset \mathbb{R}$ is infinite cyclic, equivalently, $\chi$ may be chosen to take values in $\mathbb{Q}\subset \mathbb{R}$.  The set of rational points in $S(G)$ is denoted $S_\mathbb{Q}(G)$.  We denote by $D_\mathbb{Q}(G)$ the set of isolated rational points in $\Sigma^c(G)$.  The set of all limit points of $D_\mathbb{Q}(G)$ which are 
contained in $S_\mathbb{Q}(G)$ will be denoted by 
$L_\mathbb{Q}(G)$.   
Also, we denote by $L(G)$ the set of all limit points of $\Sigma^c(G)$.  
Since $\Sigma^c(G)$ is closed,  $L_\mathbb{Q}(G)$ and $L(G)$ are subsets of $\Sigma^c(G)$.  
Any homeomorphism of $\Sigma^c(G)$ induced by an automorphism of $G$ maps $D_\mathbb{Q}(G), 
L_\mathbb{Q}(G), L(G)$ respectively onto itself.

We are now ready to prove the following theorem.   The proof is essentially the same in spirit as that of \cite[Theorem 3.3]{gk}.   See also \cite[\S 4c]{gk}.

\begin{theorem} 
Suppose that $G=G_1\times\cdots\times G_n$, $n\ge 1$, is finitely generated and that any one of the following holds: 
(i) the set $D_\mathbb{Q}(G_1)$ is non-empty, finite, and is contained in an open hemisphere and $D_\mathbb{Q}(G_j)$ is finite (possibly empty) for $2\le j\le n$,
(ii) the set $L_\mathbb{Q}(G_1)$ is non-empty, finite, and is contained in an open hemisphere and $L_\mathbb{Q}(G_j)$ is a finite set (possibly empty) for $2\le j\le n$, (iii) the set $L(G_1)\cap S_\mathbb{Q}(G_1)$ is a non-empty finite set contained 
in an open hemisphere and the set $L(G_j)\cap S_\mathbb{Q}(G_j)$ is finite (possibly empty) for $2\le j\le n$. 
Then $G$ has the $R_\infty$-property.
\end{theorem} 

\begin{proof} 
Let $\phi\in \aut(G)$.  We shall show that there exists a discrete character $\lambda\in \hom(G,\mathbb{R})$  
such that $\lambda\circ \phi=\lambda$. 
By the discussion in \S\ref{sigmath}, it follows that $R(\phi)=\infty$ and it follows that $G$ has the $R_\infty$-property.

First we suppose that $n=1$.  The theorem, then, is essentially due to Gon\c{c}alves and Kochloukova \cite{gk}.  
Let $\phi\in \aut(G)$ and let $\phi^*:\Sigma^c(G)\to \Sigma^c(G)$ be the induced map, defined as $\phi^*([\chi])=[\chi\circ\phi]$.     The map $\phi^*:\Sigma^c(G)\to \Sigma^c(G)$ being a homeomorphism, it maps isolated points to isolated points.   Moreover, $\phi^*$ preserves the set of all rational points in $\Sigma^c(G)$.  
It follows that $\phi^*(W)=W$ where $W$ is one of the sets $D_\mathbb{Q}(G), L_\mathbb{Q}(G)$ or $L(G)\cap S_\mathbb{Q}(G)$. 

In each of the cases (i)-(iii), we see that there is a 
non-empty finite set of rational character classes $W(G)\subset S_\mathbb{Q}(G)$ which is contained in an open hemisphere and is mapped to itself by $\phi^*$.  Suppose that $[\chi]\in W(G)$ and that the orbit of $[\chi]$ under
$\phi^*,$ namely the set $\{(\phi^*)^j([\chi]=[\chi\circ\phi^j]\mid j\in \mathbb{N}\}$, has $k$ elements.  Then the  
orbit sum $\lambda:=\sum_{0\le j< k}\chi\circ \phi^j\in \hom(G,\mathbb{R})$ is a non-zero discrete character 
invariant under $\phi^*$, as was to be shown.

Now let $n=2$.  
By Meinert's theorem (Theorem \ref{meinert}) $D_\mathbb{Q}(G)=D_\mathbb{Q}(G_1)\sqcup D_\mathbb{Q}(G_2)$, 
$L_\mathbb{Q}(G)=L_\mathbb{Q}(G_1)\sqcup L_\mathbb{Q}(G_2)$ and $L(G)=L(G_1)\sqcup L(G_2)$.

Case (i).  Let $[\chi]\in D_\mathbb{Q}(G_1)$. Consider the $\phi^*$-orbit of $[\chi]$, namely, 
$\{(\phi^k)^*([\chi])=[\chi\circ\phi^k]\mid k\in \mathbb{Z}\}$.  This set is finite since it is contained in $D_\mathbb{Q}(G)=D_\mathbb{Q}(G_1)\sqcup D_\mathbb{Q}(G_2)$, which is finite.   Suppose that $[\chi\circ \phi^j], 0\le j<q,$ are the distinct 
rational points in the orbit.   
We claim that the orbit sum $\lambda:=\sum_{0\le j<q}\chi\circ \phi^j$ is a {\it non-zero} character such that 
$\lambda\circ \phi=\lambda$.   To see that $\lambda\in \hom(G,\mathbb{R})$ is non-zero, note that its restriction to $G_1$ 
is the character $\lambda_J=\sum_{j\in J} \chi\circ \phi^j$ where $J:=\{j<q\mid [\chi\circ\phi^j]\in D_\mathbb{Q}(G_1)\}$. 
Since $D_\mathbb{Q}(G_1)$ is contained in an open hemisphere, the characters $\chi\circ \phi^j, j\in J,$ are in an 
open half-space of $\hom(G_1,\mathbb{R})$.  Therefore the same is true of their sum, $\lambda_1$, and we conclude that 
$\lambda\ne 0$.  It is clear that $\lambda\circ \phi=\lambda$ since $[\lambda\circ\phi]=[\lambda]$ and $\lambda$ is 
rational.  As observed in the first para of \S\ref{sigmath}, this implies that $R(\phi)=\infty$.  

Proof in cases (ii) is almost identical, starting with a $[\chi]\in L_\mathbb{Q}(G_1)$. We need only observe that 
$\phi^*(L_\mathbb{Q}(G))=L_\mathbb{Q}(G)$ and that as in case (ii), $L_\mathbb{Q}(G)=L_\mathbb{Q}(G_1)\sqcup 
L_\mathbb{Q}(G_2)$ is finite.   The orbit sum $\lambda:=\sum_{0\le j<q} \chi\circ \phi^j$ is again a non-zero 
character which is discrete and satisfies $\lambda\circ \phi=\lambda$.  Again we conclude that $R(\phi)=\infty$. 

Case (iii).  Again we start with a $\chi\in L(G_1)\cap S_\mathbb{Q}(G_1)$ and proceed as in case (ii).  We leave  
the details to the reader.

Finally, let $n\ge 3$ be arbitrary.  Let $H=G_2\times\cdots\times G_n$.   Again by Meinert's theorem 
$D_\mathbb{Q}(H)=\sqcup_{2\le j\le n}D_\mathbb{Q}(G_j)$ and similar expressions hold for $L_\mathbb{Q}(H)$ and $L(H)\cap S_\mathbb{Q}(H)$.  Our hypotheses on $G_j$ implies that one of the sets $D_\mathbb{Q}(H), L_\mathbb{Q}(H), L(G)\cap S_\mathbb{Q}(G)$ is finite depending on case (i), (ii), and (iii) respectively.  Since $G=G_1\times H$, 
we are now reduced to the situation where $n=2$, which has just been established. This completes the 
proof.
\end{proof}
  
We conclude the paper with the following examples.

\begin{examples}
(i)  Examples of groups with $D_\mathbb{Q}(G)$ non-empty, finite, and contained in an open hemisphere 
are known.    These include non-polycyclic nilpotent-by-finite groups of type $FP_\infty$, 
the generalized Richard Thompson groups $F_{n,\infty}$, the double of a knot group $K$ with non-finitely 
generated commutator subgroup (thus $G\cong K\star_{\mathbb{Z}^2}K$). 
For details see \cite[\S4]{gk}.

(ii) Examples of groups with $D_\mathbb{Q}(G)$ and $L_\mathbb{Q}(G)$ being finite sets are finite groups, the Houghton groups \cite{brown}, the pure symmetric automorphism groups \cite{o}, finitely generated infinite groups with finite abelianization (which include the generalized Richard Thompson groups $T_{n,r}$; see \cite[p. 64]{brown}),  $\mathbb{Z}^n, n\ge 1$, and 
the free groups of rank $n\ge 2$.   Another class of such groups is provided by \cite[Theorem 8.1]{bns}.  Consider 
a finitely generated 
group $G$ which is a subgroup of the group of all orientation preserving PL-homeomorphisms of the interval $[0,1]$.  The group $G$ is said to be 
{\it irreducible} if there is no $G$ fixed point in $(0,1)$.  The logarithms of the slopes near the end points $0,1$, 
define characters $\chi_0,\chi_1: G\to \mathbb{R}$ respectively.  We recall that two characters $\lambda, \chi$ are independent if $\lambda(\ker(\chi))=\lambda(G)$ and $\chi(\ker(\lambda))=\chi(G)$.  
It was shown in \cite[Theorem 8.1]{bns} that $\Sigma^c(G)=\{[\chi_0],[\chi_1]\}$ if $G$ is irreducible and $\chi_0,\chi_1$ 
are independent.   (These points may not be in $S_\mathbb{Q}(G)$; see \cite[p. 470]{bns}. ) 

(iii) Let $G=G_1\times G_2$ where $G_1$ is a finite product of groups (with $G_1$ non-trivial) as in example (i) 
and $G_2$, a finite product of groups as in example (ii) above.  Then $G$ has the $R_\infty$-property.  Since there are continuously many pairwise 
non-isomorphic two generated infinite simple groups, taking $G_2$ to be any one of them, we obtain a 
continuous family of groups with $R_\infty$-property.
\end{examples}

\noindent
{\bf Acknowledgments:}  We thank Dessislava Kochloukova for pointing out to us the paper \cite{cox} 
concerning the Houghton groups.

\end{document}